\documentclass{amsart}
\usepackage{graphicx}
 \usepackage{stmaryrd}
\usepackage{color}
\def\({\left (}
\def\){\right )}
\def\<{\left\langle}
\def\>{\right\rangle}

 \newtheorem{thm}{Theorem}[section]

\newtheorem{lem}[thm]{Lemma}

\newtheorem{defn}[thm]{Definition}
\newtheorem{rem}[thm]{Remark}

\newcommand{\norm}[1]{\left\Vert#1\right\Vert}
\newcommand{\abs}[1]{\left\vert#1\right\vert}
\newcommand{\set}[1]{\left\{#1\right\}}
\newcommand{\Real}{\mathbb R}

\newcommand{\pfrac}[2]{\frac{\partial #1}{\partial #2}}

\begin{document}
\title{On the finite time blow-up of biharmonic map flow in dimension four}

\numberwithin{equation}{section}
\author{Lei Liu and Hao Yin}

\address{School of Mathematical Sciences,
University of Science and Technology of China, Hefei, China}
\email{LLEI1988@mail.ustc.edu.cn}
\email{haoyin@ustc.edu.cn }

\begin{abstract}
	In this paper, we show that for certain initial values, the (extrinsic) biharmonic map flow in dimension four must blow up in finite time.
\end{abstract}

\subjclass{58E20(35J50 53C43)} \keywords{biharmonic map flow, finite blow-up, neck analysis.}

 \maketitle

 \pagestyle{myheadings} \markright {finite time blow-up}

\section{Introduction}

Let $(M,g)$ be a closed Riemannian manifold of dimension four and $(N,h)$ be another closed Riemannian manifold, which is isometrically embedded in $\Real^N$. The critical points of the following functional
\begin{equation*}
	E(u)=\int_M \abs{\triangle u}^2 dv
\end{equation*}
are called (extrinsic) biharmonic maps.
We also define
\begin{equation*}
	\mathcal E(u)=\int_M \abs{\nabla^2 u}^2 +\abs{\nabla u}^4 dv
\end{equation*}
and notice that since the target manifold is compact, we can bound $\mathcal E(u)$ by $E(u)$.

The associated heat flow of $E(u)$ was first studied by Lamm \cite{Lamm}. In \cite{Lamm}, the author proved that in dimension four, the following evolution equation
\begin{equation}\label{eqn:flow}
	\partial_t u =-\triangle^2 u +\triangle(B(u)(\nabla u,\nabla u)) +2\nabla \langle\triangle u\nabla P(u)\rangle -\langle \triangle P(u),\triangle u \rangle
\end{equation}
has a local solution for all smooth initial value. Here $B$ is the second fundamental form of $N\subset \Real^N$ and $P(u)$ is the projection to the tangent space $T_uN$. Moreover, the solution is global if the $W^{2,2}$ norm of the initial value is small.  Following the famous work of Struwe on harmonic map flow \cite{Struwe}, Gastel \cite{Ga} and Wang \cite{W07} showed the existence of a global weak solution with at most finitely many singular times.

It is a natural question whether the flow develops finite singularity. The problem is particularly interesting given that all weak biharmonic maps with bounded $W^{2,2}$ norm in dimension four are known to be smooth (see \cite{W04}). The corresponding problem for harmonic map flow was answered by Chang, Ding and Ye \cite{CDY}. After that, more finite-time singularity examples were found by Topping \cite{topping}, Li and Wang \cite{LW} and very recently by Chen and Li \cite{CL}. The last construction shows that the blow-up could be forced by topological reason and its proof relies on the no neck theorem for approximate harmonic maps of Qing and Tian \cite{QT}. In fact, it was pointed out by Qing and Tian that the no neck theorem could be used in showing finite time blow-up.

Recently, the authors proved the no neck theorem for the blow-up of a sequence of (extrinsic) biharmonic maps with bounded energy. In light of \cite{CL}, it is very natural to move the argument to the case of biharmonic map flow and this is the purpose of this paper. Precisely, we show
\begin{thm}
	\label{thm:main}
	Suppose that $M'$ is any closed manifold of dimension $m>4$ with nontrivial $\pi_4(M')$ and let $M=M' \# T^m$ be the connected sum of $M'$ with the torus of the same dimension. For any Riemannian metric $g$ on $M$, we can find (infinitely many) initial map $u_0:S^4\to M$ such that the biharmonic map flow (\ref{eqn:flow}) starting from $u_0$ develops finite time singularity.
\end{thm}

As remarked earlier, the proof relies on the idea of \cite{CL}. However, we give a slightly different presentation. Since we are less ambitious in proving the most general theorems, our assumption on the topology of $M$ enables us to be more specific in the construction. Moreover, we define and use the concept of the width of a biharmonic map $u$ from $S^4$ to $M$. Very roughly, the idea of the proof is the following. By a compactness argument, we show that the width of biharmonic maps from $S^4$ to $M$ is bounded by a constant depending on the energy of the map (and the geometry of $M$ of course). However, we can construct initial map $u_0$ with bounded energy but in a homotopy class in which every smooth representation must have very large width. If no finite-time singularity occurs, we may choose a sequence $t_i\to \infty$ such that the bi-tension field of $u(t_i)$ goes to zero in $L^2$ norm. Hence, $u(t_i)$ is a sequence of approximate biharmonic maps. $u(t_i)$ either converges to a smooth biharmonic map in the same homotopy class, which is not possible because the energy of the limit is smaller than that of $u_0$, or blows up. In the latter case, the total number and energy of each bubble, as well as the weak limit is bounded and the no neck theorem (Theorem \ref{thm:noneck}) implies a contradiction as well.

The rest of the paper is organized as follows. In Section \ref{sec:neck}, we generalize the no neck result in \cite{ours} to the case of approximate biharmonic maps. The generalization is in two directions. The first is to involve a non-zero bi-tension field and the second is to show the neck analysis works on round sphere instead of flat domains in $\Real^4$.

\begin{rem}
	For many PDE theorems, especially about regularity of geometric PDE, the curvature of the domain is not essential. Hence, it suffices to prove the theorem in the case of domains of Euclidean space. In this paper, we think it may not be very obvious that the neck analysis of biharmonic maps works on curved space. Hence, we present a detailed proof in the case of round metric on $S^4$, which is needed by the proof of Theorem \ref{thm:main}.

	In spite of the complexity caused by the round metric, we still believe that the neck analysis works in general. However, that would require greater efforts. We also note that this is not an issue for the neck analysis of harmonic maps, because of the conformal invariance.
\end{rem}

In Section \ref{sec:width}, the width of a map from $S^4$ to $M$ is defined and the width of biharmonic maps from both $S^4$ and $\Real^4$ are bounded by the energy. Finally, Theorem \ref{thm:main} is proved in Section \ref{sec:proof}.

\begin{rem}
	Recently, we notice that Breiner and Lamm \cite{BLamm} proved a no neck theorem for a sequence of biharmonic maps with bi-tension fields in $L\log L$ when the target manifold is a sphere. In this paper, by approximate biharmonic maps, we mean bi-tension field is bounded in $L^2$.
\end{rem}

\section{No neck for approximate biharmonic maps}\label{sec:neck}
In this section, we show that the main result of \cite{ours} can be generalized to a sequence of approximate biharmonic maps $u_i$ defined on $S^4$.

We use a subscript $g$ to denote operators defined on $S^4$ with round metric, such as $\triangle_g$ and $\nabla_g$. $\triangle$ and $\nabla$ are reserved for the Laplace and gradient with respect to the flat metric given by normal coordinates around some point in $S^4$. We always take the normal coordinates $x$ so that the scaling $u(\lambda x)$ is well defined for small $\lambda$. Moreover, due to the Gauss Lemma, the geodesic ball $B_r$ is the same as the ball of radius $r$ with respect to the flat metric given by the normal coordinates. Finally, there is no need to distinguish the $L^p$ norm for our purpose.

We will prove
\begin{thm}
	\label{thm:noneck}
	Let $u_i$ be a sequence of approximate biharmonic maps from ${B}^4$ to $N$ satisfying
\begin{equation}\label{eqn:ELex}
	\triangle_g^2 u= \triangle_g( B(u)(\nabla_g u, \nabla_g u)) +2 \nabla_g \cdot \langle \triangle_g u, \nabla(P(u))\rangle -\langle \triangle_g (P(u)),\triangle_g u \rangle+\tau(u).
\end{equation} with
	\begin{equation}\label{eqn:totalenergy}
		\int_{\tilde{B}_1} \abs{\nabla_g^2 u_i}^2 +\abs{\nabla_g u_i}^4 dv_g <\Lambda\quad and \quad \|\tau(u_i)\|_{L^p(B_1)}<\Lambda
	\end{equation}
	for some $\Lambda>0$ and $p\geq\frac{4}{3}$.
	Assume that there is a positive sequence  $\lambda_i\to 0$ such that
	\begin{equation*}
		u_i(\lambda_i x)\to \omega
	\end{equation*}
	on any compact set $K\subset \Real^4$, that $u_i$ converges weakly in $W^{2,2}$ to $u_\infty$ and that $\omega$ is the only bubble. Then,
	\begin{equation}\label{eqn:noneck2}
		\lim_{\delta\to 0} \lim_{R\to \infty} \lim_{i\to \infty} \mbox{osc}_{B_\delta(0)\setminus B_{\lambda_i R}(0)} u_i =0.
	\end{equation}
\end{thm}

\begin{rem}
	In Theorem \ref{thm:noneck}, we assume that there is only one bubble. The same result holds in the case of multiple bubbles. The proof is routine argument by now and hence is omitted.
\end{rem}

The proof is similar to the proof of Theorem 1.1 in \cite{ours}, which we outline below.
We first recall some definitions and results, which are modified only slightly.

\subsection{minor modifications}

The following is a modified version of $\varepsilon-$regularity, proved in the Appendix of \cite{ours}.
\begin{thm}[$\varepsilon_0$-regularity]\label{thm:regularity}
Let $u\in W^{4,p}(B_1)(p>1)$ be an approximate biharmonic map. There exists $\varepsilon_0>0$ such that if $\int_{B_1}|\nabla^2 u|^2+|\nabla u|^4dx\leq\epsilon_0$ then
\begin{eqnarray*}
\|u-\overline{u}\|_{W^{4,p}(B_{1/2})}\leq C(\|\nabla^2u\|_{L^2(B_1)}+\|\nabla u\|_{L^4(B_1)}+\|\tau(u)\|_{L^p(B_1)}),
\end{eqnarray*}
where $\overline{u}$ is the mean value of $u$ over $B_1$.
\end{thm}
\begin{rem}
	We may very well use $\nabla_g$ in the above lemma. It is the type of result that Riemannian metric does not make any difference.
\end{rem}

Next, we modify the definition of $\eta-$approximate biharmonic map as follows.
\begin{defn}\label{defn:approx}
Let $u$ be a smooth function defined on $B_{r_2}\setminus B_{r_1}$, $u$ is called an $\eta-$approximate biharmonic function if it satisfies
\begin{eqnarray}\label{eqn:approx}
\triangle_g^2u (r,\theta)&=& a_1\nabla_g\triangle_g u+a_2\nabla_g^2 u+a_3\nabla _gu+a_4 u \\ \nonumber
&& +
\frac{1}{\abs{\partial B_r}}\int_{\partial B_r}{b_1\nabla_g\triangle_g u+b_2\nabla_g^2 u+b_3\nabla_g u+b_4 u}d\sigma+h(x).
\end{eqnarray}
where $a_i,b_i$ and $h$ are smooth functions satisfying, for any $\rho\in [r_1,r_2/2]$,

(a) $\norm{g_{ij}(\rho x)-\delta_{ij}}_{C^4(B_2\setminus B_1)}<\eta$. Namely, the metric after scaling to $B_2\setminus B_1$ is close to the flat metric in $C^4$ norm.

(b)
\begin{equation}
\||x|^{4(1-1/p)}h\|_{L^p(B_{r_2}\setminus B_{r_1})}\leq\eta
	\label{eqn:h}
\end{equation}

and

(c)
\begin{equation*}
	\sum_{i=1}^4 \norm{a_i}_{L^{4/i}(B_{2\rho}\setminus B_{\rho})} +\norm{b_i}_{L^{4/i}(B_{2\rho}\setminus B_{\rho})} \leq \eta.
\end{equation*}
\end{defn}

\begin{rem}
	One can check that if $u$ is an $\eta-$approximate biharmonic function on $B_{r_2}\setminus B_{r_1}$, then $w(x)=u(\frac{x}{\lambda})$ is another $\eta-$approximate biharmonic function on $B_{\lambda r_2}\setminus B_{\lambda r_1}$.
\end{rem}

The following is a version of interior $L^p$ estimate for approximate biharmonic function. It is used in the proof of three circle lemma.
\begin{lem}
	\label{lem:lp}
	Suppose that $u: B_4\setminus B_1\to \mathbb{R}$ is a $\eta-$approximate biharmonic function(for small $\eta$) with
	\begin{equation*}
		\sum_{i=1}^4 \norm{a_i}_{L^{4/i}(B_4\setminus B_1)} +\norm{b_i}_{L^{4/i}(B_4\setminus B_1)}\leq \eta \quad \mbox{and}\quad  \norm{h}_{L^p(B_4\setminus B_1)}\leq C.
	\end{equation*}
	Then, for any $p>1$, we have
	\begin{equation*}
		\norm{u}_{W^{4,p}(B_3\setminus B_2)}\leq C (\norm{u}_{L^p(B_4\setminus B_1)}+\norm{h}_{L^p(B_4\setminus B_1)}).
	\end{equation*}
\end{lem}
\begin{proof}
Without loss of generality, we assume the metric $g$ is the standard Euclidean metric. The main idea is similar to the lemma 3.3 in \cite{ours}, but the assumptions on $a_i$ and $b_i$ are different from \cite{ours}. Next, we sketch the proof here.

	For $0< \sigma<1$, set $A_\sigma= B_{3+\sigma}\setminus B_{2-\sigma}$ and $A'_\sigma= B_{3+\frac{1+\sigma}{2}}\setminus B_{2-\frac{1+\sigma}{2}}$. Let $\varphi$ be a cut-off function supported in $A'_\sigma$ satisfying: (1) $\varphi\equiv 1$ in $A_\sigma$; (2) $\abs{\nabla^j \varphi}\leq \frac{c}{(1-\sigma)^j}$ for $j=1,2,3,4$ and some universal constant $c$; (3) $\varphi$ is a function of $\abs{x}$.

Computing directly, we have
\begin{eqnarray*}
\triangle^2(\varphi u)&=&\triangle(\varphi\triangle u+2\nabla\varphi\nabla u+u\triangle\varphi)\\
&=&\varphi\triangle^2 u+4\nabla\triangle u\nabla\varphi+4\nabla^2u\nabla^2\varphi+2\triangle u\triangle\varphi+4\nabla\triangle \varphi\nabla u+\triangle^2\varphi u\\
&=&  \varphi a_1 \nabla\triangle u + \varphi  a_2 \nabla^2  u + \varphi a_3 \nabla  u +\varphi  a_4 u+\varphi h \\
&& + \varphi \frac{1}{\abs{\partial B_r}}\int_{\partial B_r} b_1 \nabla\triangle u + b_2 \nabla^2 u + b_3 \nabla  u + b_4  u d\sigma \\
&& + 4\nabla\triangle u\nabla\varphi+4\nabla^2u\nabla^2\varphi+2\triangle u\triangle\varphi+4\nabla\triangle \varphi\nabla u+\triangle^2\varphi u.
\end{eqnarray*}
Next, we estimate the $L^p (p>1)$ norm of the right hand side of the above equation. By our choice of $\varphi$ and the assumption of $a_1$, we have
\begin{equation*}
	\norm{\nabla \triangle u \nabla \varphi}_{L^p(A'_\sigma)}  \leq \frac{C}{1-\sigma} \norm{\nabla^3 u}_{L^p(A'_\sigma)}
\end{equation*}
and
\begin{eqnarray*}
	 &&\norm{\varphi a_1 \nabla \triangle u}_{L^p(A'_\sigma)}\\
 &\leq&
 \norm{ a_1 \nabla \triangle(\varphi u)}_{L^p(A'_\sigma)}+\norm{ a_1 \nabla^2 u \nabla\varphi}_{L^p(A'_\sigma)}
 +\norm{ a_1 \nabla u \nabla^2\varphi}_{L^p(A'_\sigma)}+\norm{ a_1  u \nabla^3\varphi}_{L^p(A'_\sigma)}\\
 &\leq&
 \norm{a_1}_{L^4}\norm{\nabla^3(\varphi u)}_{L^{\frac{4p}{4-p}}(A'_\sigma)}+
\frac{C}{1-\sigma} \norm{a_1\nabla^2 u}_{L^p(A'_\sigma)}+\frac{C}{(1-\sigma)^2} \norm{a_1\nabla u}_{L^p(A'_\sigma)}\\
&&+\frac{C}{(1-\sigma)^3} \norm{a_1 u}_{L^p(A'_\sigma)}\\
 &\leq&
 \eta \norm{\varphi u}_{W^{4,p}(A'_\sigma)}+C\left(\frac{\norm{\nabla^2 u}_{L^{\frac{4p}{4-p}}(A'_\sigma)}}{1-\sigma}
 +\frac{\norm{\nabla u}_{L^{\frac{4p}{4-p}}(A'_\sigma)}}{(1-\sigma)^2}+\frac{\norm{ u}_{L^{\frac{4p}{4-p}}(A'_\sigma)}}{(1-\sigma)^3} \right)\\
  &\leq&
  \eta \norm{\varphi u}_{W^{4,p}(A'_\sigma)}+C \left( \frac{\norm{\nabla^3 u}_{L^p(A'_\sigma)}}{1-\sigma} +\frac{\norm{\nabla^2 u}_{L^p(A'_\sigma)}}{(1-\sigma)^2} +\frac{\norm{\nabla u}_{L^p(A'_\sigma)}}{(1-\sigma)^3 } +\frac{\norm{u}_{L^p(A'_\sigma)}}{(1-\sigma)^4}  \right),
\end{eqnarray*}
 the last interpolation is from Sobolev embedding (Theorem 5.8 in \cite{adams})
\begin{equation*}
\norm{\nabla^k u}_{L^{\frac{4p}{4-p}}(A'_\sigma)}\leq
 C( \norm{\nabla^{k+1} u}_{L^{p}(A'_\sigma)}+\norm{\nabla^{k} u}_{L^{p}(A'_\sigma)})
\end{equation*}
where the constant is independent of $\sigma$.

Moreover, Jensen's inequality implies that
\begin{eqnarray*}
	&& \int_{A'_\sigma} \frac{\varphi^p}{\abs{\partial B_r}^p} \left( \int_{\partial B_r} b_1 \nabla \triangle u \right)^p dx \\
	&\leq& \int_{A'_{\sigma}} \varphi^p \frac{1}{\abs{\partial B_r}}\left( \int_{\partial B_r} \abs{b_1 \nabla \triangle u}^p \right) dx \\
	&\leq & C \int_{A'_\sigma} \varphi^p\abs{b_1 \nabla^3 u}^p dx.
\end{eqnarray*}
Now, the same estimate used for $\norm{\varphi a_1 \nabla\triangle u}_{L^p(A'_\sigma)}$ can be used again to get the same upper bound.

Similar argument applies to the remaining terms and gives an estimate of $L^p$ norm of $\triangle^2(\varphi u)$, if we choose $\eta$ sufficiently small, by which the $L^p$ estimate of bi-Laplace operator implies
\begin{equation*}
	\norm{\varphi u}_{W^{4,p}(A'_\sigma)}\leq C \left( \frac{\norm{\nabla^3 u}_{L^p(A'_\sigma)}}{1-\sigma} +\frac{\norm{\nabla^2 u}_{L^p(A'_\sigma)}}{(1-\sigma)^2} +\frac{\norm{\nabla u}_{L^p(A'_\sigma)}}{(1-\sigma)^3 } +\frac{\norm{u}_{L^p(A'_\sigma)}}{(1-\sigma)^4}+\norm{h}_{L^p}  \right).
\end{equation*}
In particular, we have
\begin{eqnarray*}
	(1-\sigma)^4 \norm{\nabla^4 u}_{L^p(A_\sigma)} & \leq&  C \left( (1-\sigma)^3 {\norm{\nabla^3 u}_{L^p(A'_\sigma)}} + (1-\sigma)^2 {\norm{\nabla^2 u}_{L^p(A'_\sigma)}}\right.  \\
	&& \left. + (1-\sigma) {\norm{\nabla u}_{L^p(A'_\sigma)}} + {\norm{u}_{L^p(A'_\sigma)}}+\norm{h}_{L^p}  \right).
\end{eqnarray*}
By setting
\begin{equation*}
	\Psi_j=\sup_{0\leq \sigma\leq 1} (1-\sigma)^j \norm{\nabla^j u}_{L^p (A_\sigma)}
\end{equation*}
and noting that
\begin{equation*}
	A'_\sigma=A_{\frac{1+\sigma}{2}}\quad \mbox{and}\quad 1-\sigma= 2 (1-\frac{1+\sigma}{2}),
\end{equation*}
we obtain
\begin{equation}\label{eqn:seminorm}
	\Psi_4\leq C(\Psi_3 +\Psi_2 +\Psi_1 +\Psi_0+\norm{h}_{L^p}).
\end{equation}

We claim that for $j=1,2,3$, the following interpolation inequality holds for any $\epsilon>0$,
\begin{equation*}
	\Psi_j\leq \epsilon^{4-j} \Psi_4 +\frac{C}{\epsilon^j} \Psi_0.
\end{equation*}
In fact, by the definition of $\Psi_j$, for any $\gamma>0$, there is $\sigma_\gamma\in [0,1]$ such that
\begin{eqnarray*}
	\Psi_j &\leq& (1-\sigma_j)^j \norm{\nabla^j u}_{L^p(A_{\sigma_\gamma})}+ \gamma \\
	&\leq& \epsilon^{4-j}(1-\sigma_\gamma)^4 \norm{\nabla^4 u}_{L^p(A_{\sigma_\gamma})} +\frac{C}{\epsilon^j}\norm{u}_{L^p(A_{\sigma_\gamma})} +\gamma \\
	&\leq& \epsilon^{4-j} \Psi_4 + \frac{C}{\epsilon^j} \Psi_0 +\gamma .
\end{eqnarray*}
Here we used the interpolation inequality
\begin{eqnarray}
	\norm{\nabla^j u}_{L^p(A_{\sigma_\gamma})}\leq \eta^{4-j}\norm{\nabla^4 u}_{L^p(A_{\sigma_\gamma})}+\frac{C_3}{\eta^j} \norm{u}_{L^p(A_{\sigma_\gamma})}\label{4}
\end{eqnarray}
with $\eta= \epsilon (1-\sigma_\gamma)$. We remark that the constant in the above interpolation inequality are independent of $\sigma\in [0,1]$ (see the proof of Lemma 5.6 in \cite{adams}).

By sending $\gamma$ to $0$ and choosing small $\epsilon$, we obtain from (\ref{eqn:seminorm})
\begin{equation*}
	\Psi_4\leq C\left(\Psi_0+\norm{h}_{L^p}\right),
\end{equation*}
from which our lemma follows.
\end{proof}

For the universal constant $L>0$ given in Section 3 of \cite{ours}, set
\begin{equation*}
	A_i= B_{e^{-(i-1)L}}\setminus B_{e^{-iL}}
\end{equation*}
and
\begin{equation*}
	F_i(u) = \int_{A_i} \frac{1}{\abs{x}^4} u^2 dx.
\end{equation*}

\begin{rem}
	Here is a technical issue. We use $dx$ instead of $dv_g$ in the definition of $F_i(u)$. The advantage is that $F_i(u)$ is invariant under the scaling $x\to \lambda x$. Since $g$ is close to Euclidean metric, this difference does not matter when we use $F_i(u)$ as a control of $L^2$ norm.
\end{rem}

\begin{thm}
	\label{thm:tcapprox}
	There is some constant $\eta_0>0$ such that the following is true. Assume that $u:A_{i-1}\cup A_i \cup A_{i+1} \to \Real^N$ is an $\eta_0-$approximate biharmonic function in the sense of (\ref{eqn:approx}). Suppose
 \begin{equation}\label{cond}
\max_{i-1,i,i+1}\||x|^{4(1-1/p)}h\|^2_{L^p(A_i)}\leq\eta_0F_i(u)
\end{equation}
and
	\begin{equation}\label{eqn:notheta}
		\int_{\partial B_r} u d\theta=0
	\end{equation}
	for $r\in [e^{-l_1 L}, e^{-(l_2-1)L}]$.   Then

	(a) if $F_{i+1}(u)\leq e^{-L} F_i(u)$, then $F_i(u)\leq e^{-L} F_{i-1}(u)$;

	(b) if $F_{i-1}(u)\leq e^{-L} F_i(u)$, then $F_i(u)\leq e^{-L} F_{i+1}(u)$;

	(c) either $F_i(u)\leq e^{-L} F_{i-1}(u)$, or $F_i(u)\leq e^{-L} F_{i+1}(u)$.
\end{thm}

\begin{proof}
	(The proof is almost the same as Theorem 3.4 in \cite{ours}. For reader's convenience, we repeat it below.)

	The exact value of $i$ does not matter, because $F_i$ is invariant under scaling. Hence, we consider only the case of $i=2$. Assume the theorem is not true. We have a sequence of $\eta_k\to 0$ and a sequence of $u_k$ defined on $A_1\cup A_2\cup A_3$ (and a sequence of $g_k$ defined on $A_1\cup A_2\cup A_3$ as required in (a) of Definition \ref{defn:approx}) satisfying
\begin{eqnarray}\label{eqn:kk}
	&&\triangle_{g_k}^2u_k (r,\theta)= a_{k1}\nabla_{g_k}\triangle_{g_k} u_k+a_{k2}\nabla_{g_k}^2 u_k+a_{k3}\nabla_{g_k} u_k+a_{k4} u_k \\ \nonumber
&& +
\frac{1}{\abs{\partial B_r}}\int_{\partial B_r}{b_{k1}\nabla_{g_k}\triangle_{g_k} u_k+b_{k2}\nabla_{g_k}^2 u_k+b_{k3}\nabla_{g_k} u_k+b_{k4} u_k}+h_k(x)
\end{eqnarray}
with
 \begin{equation}\label{5}
\max_{1,2,3}\||x|^{4(1-1/p)}h_k\|^2_{L^p(A_i)}\leq\eta_kF_2(u_k)
\end{equation}
and
\begin{equation}\label{eqn:kksmall}
	\sum_{i=1}^4 \norm{a_{ki}}_{L^{4/i}(B_{2\rho}\setminus B_{\rho})} +\norm{b_{ki}}_{L^{4/i}(B_{2\rho}\setminus B_{\rho})} \leq \eta_k,
\end{equation}
for any $B_{2\rho}\setminus B_{\rho}\subset A_1\cup A_2\cup A_3$.

By taking subsequence, we assume that one of (a), (b) and (c) is not true for $u_k$. If (a) is not true, then we have
\begin{equation*}
	F_2(u_k)\geq e^{L} F_3(u_k) \quad\mbox{and}\quad F_2(u_k)> e^{-L}F_1(u_k).
\end{equation*}
If (b) is not true, then
\begin{equation*}
	F_2(u_k)\geq e^{L} F_1(u_k) \quad\mbox{and}\quad F_2(u_k)> e^{-L}F_3(u_k).
\end{equation*}
If (c) is not true, then
\begin{equation*}
	F_2(u_k)> e^{-L} \max \{F_1(u_k), F_3(u_k)\}.
\end{equation*}
In any case, we control $F_1(u_k)$ and $F_3(u_k)$ by $F_2(u_k)$. Multiplying by a constant to $u_k$ if necessary, we assume that $F_2(u_k)=1$ for all $k$. The above discussion shows that
\begin{equation*}
	\norm{u_k}_{L^2(A_1\cup A_2\cup A_3)}\leq C.
\end{equation*}
Lemma \ref{lem:lp} shows that (by passing to a subsequence) we have
\begin{eqnarray*}
u_k\rightharpoonup u\quad && weakly \quad in \quad L^2(A_{1}\cup A_2\cup A_3),\\
u_k\rightarrow u \quad
&& strongly \quad in \quad L^2(A_2).
\end{eqnarray*}
By (\ref{eqn:kk}), (\ref{5}) and (\ref{eqn:kksmall}), we know that $u$ is a nonzero biharmonic function with respect to the flat metric defined on $A_1\cup A_2\cup A_3$ satisfying (\ref{eqn:notheta}), because $g_k$ converges strongly in $C^3$ norm to the flat metric. The three circle lemma for biharmonic function (Theorem 3.1 in \cite{ours}) implies that
\begin{equation}\label{eqn:good}
	2F_2(u)< e^{-L} (F_1(u)+F_3(u)).
\end{equation}
If (c) does not hold for $u_k$, we have
\begin{equation*}
	2F_2(u_k)\geq e^{-L}(F_1(u_k)+F_3(u_k)).
\end{equation*}
By the strong convergence of $u_k$ in $L^2(A_2)$ and weak convergence in $L^2(A_1\cup A_2\cup A_3)$, we have
\begin{equation*}
	2F_2(u)\geq e^{-L}(F_1(u)+F_3(u)),
\end{equation*}
which is a contradiction to (\ref{eqn:good}). Similar argument works for other cases.
\end{proof}

\subsection{estimate the tangential energy}

Let $u_i$ be the sequence in Theorem \ref{thm:noneck}. Assume that
\begin{equation*}
	\Sigma=B_\delta\setminus B_{\lambda_i R}= \bigcup_{l=l_0}^{l_i} A_l
\end{equation*}
and for any $\varepsilon>0$, by choosing $\delta$ small and $R$ large, we may also assume (by an inductiona argument of Ding and Tian \cite{DT})
\begin{equation}\label{eqn:DT}
	\int_{A_l} \abs{\nabla^2 u_i}^2 + \abs{\nabla u_i}^4 dx < \varepsilon^4 <\varepsilon_0
\end{equation}
for $l=l_0,\cdots,l_i$. Set $\widetilde{u_i}(x)=u_i(e^{-lL}x)$, by $\varepsilon_0$-regularity Theorem \ref{thm:regularity}, we have
 \begin{equation*}
 \|\widetilde{u_i}-\overline{\widetilde{u_i}}\|_{W^{4,p}(A_0)}\leq C(\|\nabla^2\widetilde{u_i}\|_{L^2(A_{-1}\cup A_0\cup A_1)}+\|\nabla \widetilde{u_i}\|_{L^4(A_{-1}\cup A_0\cup A_1)}+\|\tau(\widetilde{u_i})\|_{L^p(A_{-1}\cup A_0\cup A_1)}),
\end{equation*}
where $\overline{\widetilde{u_i}}$ is the mean value of $\widetilde{u_i}$ over $A_0$.

Scaling back, if $\delta$ is sufficiently small, we will get
\begin{eqnarray*}
&&\sum_{k=0}^4\||x|^{k-4/p}\nabla^k(u_i-\overline{\widetilde{u_i}})\|_{L^p(A_l)}\\
&\leq& C(\|\nabla^2u_i\|_{L^2(A_{l-1}\cup A_l\cup A_{l+1})}+\|\nabla u_i\|_{L^4(A_{l-1}\cup A_l\cup A_{l+1})}+e^{-lL4(1-1/p)}\|\tau(u_i)\|_{L^p(A_{l-1}\cup A_l\cup A_{l+1})})\\
&\leq& C\varepsilon.
\end{eqnarray*}
Let $r=e^t$, then as a function of $(t,\theta)$, we have
\begin{eqnarray}\label{3}
\|u_i-\overline{\widetilde{u_i}}\|_{W^{4,p}(-lL,-(l-1)L)\times S^3}\leq C\varepsilon,
\end{eqnarray}
for any $l_0\leq l\leq l_i$.

 The theorem is equivalent to the statement that for any $\varepsilon>0$, we can find $\delta$ small and $R$ large such that
\begin{equation*}
	\mbox{osc}_{B_\delta\setminus B_{\lambda_i R}}u_i <C\varepsilon
\end{equation*}
for $i$ sufficiently large.

Set
\begin{equation*}
	u_i^*(r)=\frac{1}{\abs{\partial B_r}} \int_{\partial B_r} u_i(r,\theta) d\sigma.
\end{equation*}
The Poincar\'e inequality and (\ref{eqn:DT}) imply
\begin{equation*}
	\int_{A_l} \frac{1}{\abs{x}^4} \abs{u_i -u_i^*}^2 dx \leq C \varepsilon^2.
\end{equation*}

\begin{lem}
	There exists some $\varepsilon_1>0$ such that if $\varepsilon<\varepsilon_1$ in (\ref{eqn:DT}) and $\delta< \varepsilon_1$, $w_i=u_i-u_i^*$ is an $\eta_0-$approximate biharmonic function defined on $B_\delta\setminus B_{\lambda_i R}$ in the sense of (\ref{eqn:approx}), where $\eta_0$ is the constant in Theorem \ref{thm:tcapprox}.
\end{lem}

\begin{rem}
	Although the proof is parallel to Lemma 4.1 in \cite{ours}. We reproduce it because (1) we now uses the sphere metric instead of the flat one; (2) the definition of $\eta-$approximate biharmonic function is different.
\end{rem}

\begin{proof}
	For simplicity, we omit the subscript $i$.
Recall that u satisfies
\begin{eqnarray}\label{eqn:euler}
	&& \triangle_g^2 u = \alpha_1(u) \nabla_g \triangle_g u\# \nabla_g u+ \alpha_2(u) \nabla_g^2  u\# \nabla_g^2 u \\ \nonumber
&&	+ \alpha_3(u) \nabla_g^2 u \# \nabla_g u \# \nabla_g u + \alpha_4(u)\nabla_g u \# \nabla_g u \#\nabla_g u \# \nabla_g u + \tau(u).
\end{eqnarray}
Here $\alpha_i(u)$ is a smooth function of $u$ and $\#$ is the contraction of tensors with respect to $g$, for which we have for example,
\begin{equation*}
	\abs{\nabla_g \triangle_g u\# \nabla_g u}\leq C \abs{\nabla_g \triangle_g u} \abs{\nabla_g u}.
\end{equation*}
Since $\triangle_g = \frac{\partial^2}{\partial r^2} + \frac{3\cos r}{\sin r}\pfrac{}{r} +\frac{1}{\sin^2 r}\triangle_{S^3}$ and $\int_{S^3} \triangle_{S^3} f d\theta=0$ for any $f$, we have
\begin{eqnarray*}
	\triangle_g^2 u^*(r)&=& \frac{1}{\abs{\partial B_r}} \int_{\partial B_r} \triangle_g^2 u d\sigma \\
	&=& \frac{1}{\abs{\partial B_r}}\int_{\partial B_r} \alpha_1(u) \nabla_g \triangle_g u\# \nabla_g u+ \alpha_2(u) \nabla_g^2  u\# \nabla_g^2 u \\
&&	+ \alpha_3(u) \nabla_g^2 u \# \nabla_g u \# \nabla_g u + \alpha_4(u)\nabla_g u \# \nabla_g u \#\nabla_g u \# \nabla_g u d\sigma \\
&& +\frac{1}{\abs{\partial B_r}}\int_{\partial B_r} \tau(u) d\sigma \\
&=& I + II + III + IV +\frac{1}{\abs{\partial B_r}}\int_{\partial B_r} \tau(u) d\sigma .
\end{eqnarray*}

\begin{rem}
	Here we make essential use of the symmetry of spherical metric to simplify the computation in the first line above. This is partially the reason that we work on round $S^4$.
\end{rem}

Computing directly, we get
\begin{eqnarray*}
	I &=& \frac{1}{\abs{\partial B_r}} \int_{B_r} \alpha_1(u) \nabla_g \triangle_g u \# \nabla_g u -\alpha_1(u^*) \nabla_g \triangle_g u \# \nabla_g u \\
	&&  + \alpha_1(u^*) \nabla_g \triangle_g u \# \nabla_g u -\alpha_1(u^*) \nabla_g \triangle_g u^* \# \nabla_g u \\
	&& + \alpha_1(u^*) \nabla_g \triangle_g u^* \#\nabla_g u -\alpha_1(u^*) \nabla_g \triangle_g u^* \# \nabla_g u^* d\sigma \\
	&& + \alpha_1(u^*) \nabla_g \triangle_g u^* \# \nabla_g u^* \\
	&=& \frac{1}{\abs{\partial B_r}} \int_{\partial B_r} \beta_4[u] (u-u^*) + \beta_1[u] \nabla_g \triangle_g (u-u^*)\\
	&& + \beta_3[u] \nabla_g (u-u^*) d\sigma + \alpha_1(u^*) \nabla_g \triangle_g u^* \# \nabla_g u^*.
\end{eqnarray*}
Here $\beta_i[u]$ is some expression depending on $u$, $u^*$ and their derivatives. Those $\beta_i$'s may differ from line to line in the following. However, thanks to Theorem \ref{thm:regularity}, we have
\begin{equation*}
	\norm{\beta_i}_{L^{4/i}(B_{2\rho}\setminus B_{\rho})}\leq \eta_0 \quad \mbox{for}\quad \rho\in [\lambda_i R,\delta/2],
\end{equation*}
if $\varepsilon$ in (\ref{eqn:DT}) is smaller than some $\varepsilon_1$. We shall require the above holds for all $\beta_i$ and $\beta'_i$ below by asking $\varepsilon_1$ to be smaller and smaller.

The same computation gives
\begin{eqnarray*}
	II&=& \frac{1}{\abs{\partial B_r}} \int_{B_r} \beta_4[u] (u-u^*) + \beta_2[u] \nabla_g^2 (u-u^*) d\sigma + \alpha_2(u^*) \nabla_g^2 u^* \# \nabla_g^2 u^*,
\end{eqnarray*}
\begin{eqnarray*}
	III&=&  \frac{1}{\abs{\partial B_r}} \int_{\partial B_r}\beta_4[u] (u-u^*) + \beta_2[u] \nabla_g^2 (u-u^*) + \beta_3[u] \nabla_g (u-u^*) d\sigma \\
	&& + \alpha_3(u^*) \nabla_g^2 u^* \# \nabla_g u^* \# \nabla_g u^*
\end{eqnarray*}
and
\begin{eqnarray*}
	IV &=& \frac{1}{\abs{\partial B_r}} \int_{\partial B_r} \beta_4[u] (u-u^*) + \beta_3[u] \nabla_g (u-u^*) d\sigma + \alpha_4(u^*) \nabla_g u^* \#\nabla_g u^* \#\nabla_g u^* \#\nabla_g u^*.
\end{eqnarray*}
In summary, $u^*$ satisfies an equation similar to (\ref{eqn:euler}) except an error term of the form
\begin{equation*}
	\frac{1}{\abs{\partial B_r}} \int_{\partial B_r} \beta_1[u] \nabla_g \triangle_g w + \beta_2[u] \nabla_g^2 w + \beta_3[u] \nabla_g w + \beta_4[u] w d\sigma.
\end{equation*}
Subtract the equation of $u^*$ with (\ref{eqn:euler}) and handle the terms like $\alpha_1(u)\nabla_g \triangle_g u \# \nabla_g u- \alpha_1(u^*) \nabla_g \triangle_g u^* \# \nabla_g u^*$ as before to get
\begin{eqnarray*}
	\triangle_g^2 w &=& \beta'_1[u] \nabla_g \triangle_g w + \beta'_2 [u] \nabla_g^2  w
		+ \beta'_3[u] \nabla_g w  + \beta'_4[u] w \\ \nonumber
		&& +
	\frac{1}{\abs{\partial B_r}} \int_{\partial B_r} \beta_1[u] \nabla_g \triangle_g w + \beta_2[u] \nabla_g^2 w + \beta_3[u] \nabla_g w + \beta_4[u] w d\sigma +h,
\end{eqnarray*}
where
\begin{equation*}
	h= \tau(u)-\frac{1}{\abs{\partial B_r}}\int_{\partial B_r} \tau(u) d\sigma.
\end{equation*}
To see that $h$ satisfies (b) of Definition \ref{defn:approx}, we notice that $4(1-\frac{1}{p})>0$ and
	\begin{equation*}
		\norm{ \abs{x}^{4(1-\frac{1}{p})} h}_{L^p(B_\delta\setminus B_{\lambda_i R})} \leq \delta^{4(1-\frac{1}{p})} \norm{h}_{L^p(B_\delta\setminus B_{\lambda_i R})}\leq C \delta^{4(1-\frac{1}{p})} \norm{\tau(u_i)}_{L^p(B_1)}.
	\end{equation*}
Since $\tau(u_i)$ is uniformly bounded in $L^p$, the lemma follows by choosing $\delta$ small.
\end{proof}

Now we apply Theorem \ref{thm:tcapprox} to the function $w_i$.
\begin{lem}
	For any $0<\varepsilon<\varepsilon_1$ and sufficiently small $\delta>0$, we have
\begin{equation}\label{1}
	F_l(w_i)\leq  C\varepsilon^2\left(e^{-\min\set{8(1-1/p),1}(l-l_0)L}+e^{-\min\set{8(1-1/p),1}(l_i-l)L}\right),
\end{equation}
for $l_0<l<l_i$.
\end{lem}

\begin{proof}
	Let the set of $l(l_0<l<l_i)$, for which the condition (\ref{cond}) is not true, be denoted by $\set{j_1,\cdots,j_{n_i}}$ and we assume that
	\begin{equation*}
		l_0< j_1 < j_2 < \cdots < j_{n_i} < l_i.
	\end{equation*}
	By definition, for each $l=j_k$,
\begin{eqnarray}
\max_{l-1,l,l+1}\||x|^{4(1-1/p)}h_i\|^2_{L^p(A_l)}\geq\eta_0F_l(w_i).
\end{eqnarray}
Then we have
\begin{eqnarray*}
	F_l(w_i)&\leq& C \max_{l-1,l,l+1} \norm{ \abs{x}^{4(1-\frac{1}{p})}h_i}_{L^p(A_l)}\\
	&\leq&Ce^{-8(1-1/p)lL}\\
&\leq& C\delta^{8(1-1/p)} e^{-8(1-1/p)(l-l_0)L} \\
&\leq& C\varepsilon^2 e^{-8(1-1/p) (l-l_0)L},
\end{eqnarray*}
if we choose $\delta$ small.

By the choice of ${j_k}$, the condition (\ref{cond}) holds for ${j_k}<l<{j_{k+1}}$, $k=1,...,i-1$. 
By an application of Theorem \ref{thm:tcapprox} (see also Lemma 4.2 in \cite{ours}), we have, for ${j_k}<l<{j_{k+1}}$
\begin{eqnarray*}
F_l(w_i)&\leq& C\left( e^{-L (l-{j_k})}F_{{j_k}}(w_i)+ e^{-L({j_{k+1}}-l)}F_{{j_{k+1}}}(w_i) \right)\\
&\leq& C \varepsilon^2 \left( e^{-\min\{8(1-1/p),1\}(l-l_0)L}\right).
\end{eqnarray*}

So, if ${j_1}=l_0+1$ and ${j_{n_i}}=l_i-1$, the inequality (\ref{1}) follows immediately. If not, assuming ${j_1}>l_0+1$, by Theorem \ref{thm:tcapprox} again, we have, for $l_0<l<j_1$,
\begin{eqnarray*}
F_l(w_i)&\leq& C\left( e^{-L (l-l_0)}F_{l_0}(w_i)+ e^{-L({j_1}-l)}F_{{j_1}}(w_i) \right)\\
&\leq&C\left( e^{-L (l-l_0)}F_{l_0}(w_i)+ \varepsilon^2 e^{-\min\{8(1-1/p),1\}(l-l_0)L}\right)\\
&\leq&C\varepsilon^2e^{-\min\{8(1-1/p),1\}(l-l_0)L}.
\end{eqnarray*}

Similarly, if $j_{n_i}<l_i-1$, we have, for $j_{n_i}<l<l_i-1$,
\begin{eqnarray*}
	F_l(w_i)&\leq& C\left( e^{-L (l-{j_{n_i}})}F_{{j_{n_i}}}(w_i)+ e^{-L(l_i-l)}F_{l_{i}}(w_i) \right)\\
&\leq&C\left( \varepsilon^2 e^{-\min\{8(1-1/p),1\}(l-l_0)L}+e^{-L(l_i-l)}F_{l_{i}}(w_i)\right)\\
&\leq&C\varepsilon^2\left(e^{-\min\{8(1-1/p),1\}(l-l_0)L}+e^{-L(l_i-l)}\right)\\
&\leq&C\varepsilon^2\left(e^{-\min\{8(1-1/p),1\}(l-l_0)L}+e^{-\min\{8(1-1/p),1\}(l_i-l)L}\right).
\end{eqnarray*}
\end{proof}

Since $w_i$ satisfies (\ref{eqn:approx}), we may use Lemma \ref{lem:lp} to get estimates for the derivatives of $w_i$ and the tangential derivatives of $u_i$. In the following, $(r,\theta)$ is the polar coordinates where $\theta\in S^3$ is a point of the unit sphere. A function $u(r,\theta)$ is also considered a function of $(\tilde{t},\theta)$, where $r=e^{\tilde{t}}$. We denote the gradient operator on $S^3$ by $\nabla_{S^3}$ and the Laplacian on $S^3$ by $\triangle_{S^3}$.

\begin{rem}
	Since we have only $L^p$ norm of bi-tension fields bounded, we may not prove pointwise decay bound for tangential derivatives. Hence we need the following lemma as a replacement.
\end{rem}

\begin{lem}\label{lem:point}
\begin{eqnarray}\label{eqn:tangentdecay}
	&&\int_{(-lL,-(l-1)L)\times S^3}\left( \abs{\triangle_{S^3} u_i}^2+\abs{\partial_{\tilde{t}}\nabla_{S^3} u_i}^2\right)d\tilde{t}d\theta\\ \nonumber
	&\leq&  C\varepsilon^2\left(e^{-\min\{8(1-1/p),1\}(l-l_0)L}+e^{-\min\{8(1-1/p),1\}(l_i-l)L}\right).
\end{eqnarray}
Or equivalently,
\begin{eqnarray*}
	&&\int_{[\tilde{t},\tilde{t}+1]\times S^3}\left( \abs{\triangle_{S^3} u_i}^2+\abs{\partial_{\tilde{t}}\nabla_{S^3} u_i}^2\right)d\tilde{t}d\theta\\ \nonumber
	&\leq& C \varepsilon^2\left(  e^{-\min\set{8(1-1/p),1}(\log\delta -\tilde{t})} + e^{-\min\set{8(1-1/p),1}(\tilde{t}-\log \lambda_i R)} \right).
\end{eqnarray*}
\end{lem}

\begin{proof}
	Setting
	\begin{equation*}
		\tilde{w}(x)= w_i (e^{-(l-1)L} x),
	\end{equation*}
	we have
	\begin{eqnarray*}
		\norm{\tilde{w}}^2_{L^2(A_0\cup A_1\cup A_2)}&\leq& C (F_{l-1}(w_i)+ F_l(w_i)+ F_{l+1}(w_i))\\
		&\leq& C\varepsilon^2 \left( e^{-\min\set{8(1-1/p),1} (l-l_0)L} + e^{-\min\set{8(1-1/p),1}(l_i-l)L} \right).
	\end{eqnarray*}

	By scaling, $\tilde{w}$ satisfies
\begin{eqnarray}\label{eqn:w}
	&&\triangle_g^2\tilde{w} (r,\theta)= a_1\nabla_g\triangle_g \tilde{w}+a_2\nabla_g^2 \tilde{w}+a_3\nabla_g \tilde{w}+a_4 \tilde{w} \\ \nonumber
&& +
\frac{1}{\abs{\partial B_r}}\int_{\partial B_r}{b_1\nabla_g\triangle_g \tilde{w}+b_2\nabla_g^2 \tilde{w}+b_3\nabla_g \tilde{w}+b_4 \tilde{w}}d\theta+\tilde{h}(x).
\end{eqnarray}
Here
\begin{equation*}
	\tilde{h}(x) = e^{-4(l-1)L} h(e^{-(l-1)L} x)
\end{equation*}
and
\begin{equation*}
	h(x)=\tau(u_i) -\frac{1}{\abs{(\partial B_r)}} \int_{\partial B_r} \tau(u_i) d\sigma.
\end{equation*}

Letting $\lambda=e^{-(l-1)L}$, we have
\begin{eqnarray}
	&& \norm{\tilde{h}}_{L^p (A_0\cup A_1\cup A_2)} \label{2}\\
	&=& \left( \int_{A_0\cup A_1\cup A_2} \abs{\lambda^4 h(\lambda x)}^p dx \right)^{\frac{1}{p}}\notag \\
	&=& \lambda^{4(1-1/p)} (\int_{A_{l-1}\cup A_l \cup A_{l+1}} \abs{h(x)}^p dx)^{1/p}\notag \\
	&\leq& C e^{-4(1-1/p)(l-1)L}\notag \\
	&\leq& C \delta^{4(1-1/p)} e^{-4(1-1/p)(l-l_0)L}\notag\\
	&\leq& C\varepsilon e^{-4(1-1/p)(l-l_0)L},\notag
\end{eqnarray}
if $\delta$ is small.

Lemma \ref{lem:lp}  and the Sobolev embedding theorem imply that
\begin{eqnarray*}
	&& \int_{(-lL,-(l-1)L)\times S^3} ( \abs{\triangle_{S^3} u_i}^2 + \abs{\partial_{\tilde{t}} \nabla_{S^3} u_i}^2 ) d\tilde{t} d\theta \\
	&\leq& C \int_{A_1} (\abs{\nabla^2 \tilde{w}}^2 + \abs{\nabla \tilde{w}}^2) dx \\
	&\leq& C\varepsilon^2 \left( e^{-\min\set{8(1-1/p),1} (l-l_0)L}+ e^{-\min\set{8(1-1/p),1}(l_i-l)L} \right).
\end{eqnarray*}
\end{proof}

\subsection{proof of Theorem \ref{thm:noneck}}
With the preparations of previous subsections, we may now prove Theorem \ref{thm:noneck}. For the rest of the proof, we require $p\geq \frac{4}{3}$ and hence $\min\set{8(1-\frac{1}{p}),1}=1$.

The rest of the proof is some type of Pohozaev argument. It follows the same line of Section 5 of \cite{ours}. However, the proof there made use of the explicit expression of bi-Laplace operator in polar coordinates of $\Real^4$. Since we are now using the round metric on $S^4$, we think it is necessary to justify the reason why the proof still works. As can be seen from below, this is not obvious and the proof depends on some detailed computation.

To begin with, we define a function (for $r<1$)
\begin{equation*}
	t(r)=\int_1^r \frac{1}{\sin s}ds.
\end{equation*}
Obviously, $t'(r)=\frac{1}{\sin r}$. One may want to compare it with $\tilde{t}(r)=\log r$. In fact, we have
\begin{equation*}
	0<\tilde{t}(r)- t(r)<C \quad \mbox{for}\quad r<1
\end{equation*}
and $\tilde{t}'(r)$ is comparable with $t'(r)$.
As a consequence, the result of Lemma \ref{lem:point} can be further rewritten as (noting that $p\geq 4/3$ here)
\begin{eqnarray}\label{eqn:samepoint}
	&&\int_{[{t},{t}+1]\times S^3}\left( \abs{\triangle_{S^3} u_i}^2+\abs{\partial_{{t}}\nabla_{S^3} u_i}^2\right)d{t}d\theta\\ \nonumber
	&\leq& C \varepsilon^2\left(  e^{-(t(\delta) -{t})} + e^{-({t}-t(\lambda_i R))} \right).
\end{eqnarray}

Recall that the metric is given by $g=dr^2 +\sin ^2 r d\theta^2$. (Here $d\theta^2$ is the standard metric on the unit sphere.) To simplify the notations, we write $f(r)=\sin r$ and $f'$ is the derivative of $f$ with respect to $r$. The Laplace operator is
\begin{equation*}
	\triangle_g u = \partial_r^2 + \frac{3f'}{f} \partial_r u + \frac{1}{f^2} \triangle_{S^3} u.
\end{equation*}
By using $\partial_t =f\partial_r$, we may compute
\begin{eqnarray*}
	\triangle_g u &=& f^{-2} \left( \partial_t^2 + 2 f' \partial_t + \triangle_{S^3} \right) u.
\end{eqnarray*}

Writing $\triangle_g u= f^{-2} w$, we obtain
\begin{eqnarray*}
	\triangle^2_g u&=& f^{-2} \left( \partial_t^2  + 2f' \partial_t  +\triangle_{S^3}  \right)(f^{-2} w) \\
	&=& f^{-4}\left( \partial_t^2 + 2 f'\partial_t +\triangle_{S^3} \right) w \\
	&& + f^{-2} \left( \partial_t^2 ( f^{-2}) w + 2 \partial_t (f^{-2}) \partial_t w + 2f' \partial_t (f^{-2}) w \right)\\
	&=& f^{-4}\left( \partial_t^2 - 2 f'\partial_t +\triangle_{S^3} \right) w- 2 \frac{f''}{f^3} w \\
\end{eqnarray*}
By the definition of $w$ and $f''=-f$, we have
\begin{eqnarray*}
	\triangle^2_g u&=& f^{-4}\left( \partial_t^2 - 2 f'\partial_t +\triangle_{S^3} \right) \left( \partial_t^2  +2f'\partial_t +\triangle_{S^3} \right)u + 2 f^{-2} w \\
	&=& f^{-4}\left( (\partial_t^2 +\triangle_{S^3})^2 -4 (f'\partial_t)(f'\partial_t) \right)u \\
	&& + f^{-4} \left( (\partial_t^2 +\triangle_{S^3} )(2f'\partial_t) - (2f'\partial_t) (\partial_t^2 +\triangle_{S^3})\right)u + 2 f^{-2} w \\
\end{eqnarray*}

In comparison with the case of flat metric, $f$ causes some extra terms. It is the primary goal here to show that we can handle these extra terms properly.

\begin{equation*}
	4(f'\partial_t) (f'\partial_t) = 4(f')^2 \partial_t^2 -4 f^2 f' \partial_t,
\end{equation*}
where we used $\partial_t= f' \partial_r$ and $f''=-f$ because $f(r)=\sin r$.

Note that $\triangle_{S^3}$ commutes with $f'\partial_t$ and we compute
\begin{eqnarray*}
	&& \partial_t^2 (2f'\partial_t) -(2f'\partial_t)\partial_t^2 \\
	&=& \partial_t^2 (2f') \partial_t + 2\partial_t (2f')\partial_t^2 \\
	&=& -4f^2 f' \partial_t  -4 f^2 \partial_t^2 .
\end{eqnarray*}
In summary, we have
\begin{equation}
	\triangle^2_g u= f^{-4}\left(  (\partial_t^2+\triangle_{S^3})^2 -4 \partial_t^2 \right)u + 2 f^{-2} w,
	\label{eqn:bilaplace}
\end{equation}
where we used $(f')^2+f^2=1$.

\begin{rem}
	The first term in the above formula is almost the same as the flat case. The importance of the computation is to show the error caused by the round metric is just $f^{-2}w$. Since $w$ involves only first and second order derivatives, it can be controlled by the energy. {\it If} there is a third order derivative term here, then the proof below would fail.
\end{rem}

By the definition of $\tau$, we have
\begin{equation*}
	\int_{S^3} f^4 \triangle^2_g u \cdot \partial_t u d\theta = \int_{S^3} f^4 \tau(u) \cdot \partial_t u d\theta.
\end{equation*}
By (\ref{eqn:bilaplace}), the above is equivalent to
\begin{equation*}
	\int_{S^4} \left( (\partial_t^2 +\triangle_{S^3})^2-4 \partial_t^2 \right)u \partial_t u d\theta = \int_{S^3} (f^4 \tau(u)- 2 f^2 w) \partial_t u d\theta.
\end{equation*}
The left hand side is now completely identical to the form which is dealt with in Section 5 of \cite{ours}. For simplicity, we set
\begin{equation*}
	\tilde{\tau}(u)= \tau(u)-2 f^{-2}w=\tau(u)-2\triangle_g u.
\end{equation*}
Since $u$ has finite energy, $\tilde{\tau}(u)$ is also uniformly bounded in $L^p$ for $p\in [4/3,2]$.

The same computation as in \cite{ours} gives
\begin{eqnarray}\label{eqn:ode}
	&&\partial_t \int_{S^3}  \partial_t u \partial_t ^2 u d\theta -\int_{S^3} \frac{3}{2} \abs{\partial_t^2 u}^2 +2 \abs{\partial_t u}^2 d\theta\\ \nonumber
	&=&\int_{S^3}-\frac{1}{2}\abs{\triangle_{S^3} u}^2 + \abs{\partial_t \nabla_{S^3} u}^2+\int_{-\infty}^t\int_{S^3}f^4 \tilde{\tau} (u) \cdot \partial_t u ds d\theta.
\end{eqnarray}
We will integrate the above inequality from $t(\lambda_i R)$ to $t(\delta)$. We estimate the right hand side first.
Thanks to (\ref{eqn:samepoint}),
we have
\begin{equation*}
	\int_{t(\lambda_i R)}^{t(\delta)} \int_{S^3}-\frac{1}{2}\abs{\triangle_{S^3} u}^2 + \abs{\partial_t \nabla_{S^3} u}^2d\theta\leq C\varepsilon^2.
\end{equation*}
Transforming back to $x-$coordinates by $\partial t= f \partial r$ and $d\sigma=f^3 d\theta$, we get
\begin{eqnarray*}
	&&\abs{\int_{t(\lambda_iR)}^{t(\delta)}\int_{-\infty}^{\tilde{t}} \int_{S^3} f^4 \tilde{\tau}(u) \cdot \partial_t u d\theta ds  d\tilde{t}} \\
	&\leq&\int_{t(\lambda_iR)}^{t(\delta)} \int_{B_{r(\tilde{t})}} \abs{\tilde{\tau}(u)} \abs{f \partial_r u}  dxd\tilde{t} \\
	&\leq&\int_{\lambda_iR}^{\delta} \int_{B_{r}} \abs{\tilde{\tau}(u)} \abs{\nabla u}  dxdr \\
	&\leq&C \delta \norm{\tilde{\tau}(u)}_{L^{4/3}(B_1)} \norm{\nabla u}_{L^4(B_1)}.
\end{eqnarray*}
In summary, the integration of (\ref{eqn:ode}) yields (by taking $\sigma$ small with respect to $\varepsilon$)
\begin{eqnarray}\label{eqn:total}
	&&\int_{t(\lambda_i R)}^{t(\delta)} \int_{S^3}\frac{3}{2} \abs{\partial_t^2 u}^2 + 2\abs{\partial_t u}^2 d\theta dt \\ \nonumber
&\leq&
C \left(\int_{\set{t(\delta)} \times S^3}  |\partial_t u \partial_t ^2 u| d\theta+\int_{\set{t(\lambda_{iR})}\times S^3}  |\partial_t u \partial_t ^2 u| d\theta+\varepsilon^2\right)\\\nonumber
&\leq&C  \norm{\partial_t u}_{L^2(\{t(\delta)\} \times S^3)}\norm{ \partial_t ^2 u}_{L^2(\{t(\delta)\} \times S^3)}\\\nonumber
&&+ C\norm{\partial_t u}_{L^2(\{t(\lambda_iR)\} \times S^3)}\norm{ \partial_t ^2 u}_{L^2(\{t(\lambda_iR)\} \times S^3)}+ C \varepsilon^2\\\nonumber
&\leq& C\varepsilon^2,
\end{eqnarray}
where the last inequality comes from the (\ref{3}) and Sobolev embedding and trace theorem. In fact, we have $W^{4,p}(\Omega)$ embeds into $W^{3,2}(\Omega)$, which in turn embeds into $W^{2,2}(\partial \Omega)$.

\begin{rem}
	We remark that in fact, the argument above gives an independent proof of the energy identity in the blow up analysis of biharmonic maps with tension field in $L^p$ for some $p\geq\frac{4}{3}$.
\end{rem}
For some fixed $t_0\in [t(\lambda_i R), t(\delta)]$, set
\begin{equation*}
	F(t)=\int_{t_0-t}^{t_0+t}\int_{S^3} \frac{3}{2} \abs{\partial_t^2 u}^2 + 2\abs{\partial_t u}^2 d\theta dt.
\end{equation*}
$F$ is defined for $0\leq t\leq \min \set{t_0- t(\lambda_i) R, t(\delta) -t_0}$. Integrating (\ref{eqn:ode}) from $t_0-t$ to $t_0+t$, we obtain
\begin{eqnarray*}
	F(t)&\leq& \frac{1}{2\sqrt{3}} \left( \int_{\set{t_0-t}\times S^3} +\int_{\set{t_0+t}\times S^3} \right) \frac{3}{2}\abs{\partial_t^2 u}^2 + 2 \abs{\partial_t u}^2 d\theta\\
	&& + \int_{t_0-t}^{t_0+t} \left( \int_{S^3}-\frac{1}{2}\abs{\triangle_{S^3} u}^2 + \abs{\partial_t \nabla_{S^3} u}^2 d\theta +\int_{-\infty}^{\tilde{t}} \int_{S^3}f^4 \tilde{\tau} (u) \cdot \partial_t u ds d\theta \right)  d\tilde{t}.
\end{eqnarray*}

With the help of (\ref{eqn:samepoint}), we can have
\begin{equation*}
	\int_{t_0-t}^{t_0+t} \int_{S^3} \frac{-1}{2} \abs{\triangle_{S^3} u}^2 +\abs{\partial_t \nabla_{S^3} u}^2 d\theta ds \leq C\varepsilon^2 \left( e^{-(t(\delta)-t_0)}+ e^{-(t_0-t(\lambda_i R))} \right)e^t.
\end{equation*}

On the other hand,
\begin{eqnarray*}
	&&\abs{\int_{t_0-t}^{t_0+t}\int_{-\infty}^{\tilde{t}} \int_{S^3} f^4 \tilde{\tau}(u) \cdot \partial_t u ds d\theta d\tilde{t}} \\
	&\leq&\int_{t_0-t}^{t_0+t} \int_{B_{r(\tilde{t})}} \abs{\tau(u)} \abs{\nabla u} \abs{f} dxd\tilde{t} \\
	&\leq&C e^{\frac{1}{2}(t_0+t)} \norm{\tau(u)}_{L^{4/3}(B_1)} \norm{\nabla u}_{L^4(B_1)}\\
	&\leq& C \delta^{1/2} e^{-1/2(\log \delta -t_0)} e^{t/2}\\
	&\leq& C \delta^{1/2} e^{-1/2(t(\delta) -t_0)} e^{t}\\
\end{eqnarray*}
\begin{rem}
	Note that since $r'(t)=\sin r$ and $\frac{1}{2}r\leq \sin r\leq r$ for $r<1$, we have
	\begin{equation*}
		e^t< r(t)< e^{t/2}
	\end{equation*}
	for $t<0$.
\end{rem}
Hence, if $\delta$ is small, we obtain
\begin{equation*}
	F(t)\leq \frac{1}{2}\partial_t F(t) + C\varepsilon^2 \left( e^{-\frac{1}{2}(t(\delta)-t_0)} + e^{-\frac{1}{2}(t_0-t(\lambda_i R))} \right)e^t.
\end{equation*}
Multiplying $e^{-2t}$ to both sides of the inequality, we have
\begin{equation*}
	(e^{-2t} F(t))'\geq -C\varepsilon^2 \left( e^{- \frac{1}{2}(t(\delta) -t_0)}+ e^{- \frac{1}{2}(t_0-t(\lambda_i R))} \right) e^{-t}.
\end{equation*}
We assume without loss of generality that $t( \delta)- t_0\leq t_0-t(\lambda_i R)$. Then, we integrate the above inequality from $t=1$ to $t=t(\delta) -t_0$ to get
\begin{eqnarray*}
	F(1)&\leq& e^{-2(t(\delta) -t_0)+2} F(t(\delta)- t_0) + C\varepsilon^2 \left( e^{- \frac{1}{2}(t(\delta) -t_0)}+ e^{- \frac{1}{2}(t_0-t(\lambda_i R))} \right) \\
	&\leq&  C\varepsilon^2 \left( e^{- \frac{1}{2}(t(\delta) -t_0)}+ e^{- \frac{1}{2}(t_0-t(\lambda_i R))} \right).
\end{eqnarray*}
Here we used (\ref{eqn:total}).

Together with (\ref{eqn:samepoint}), we obtain 
\begin{equation*}
	\int_{t_0-1}^{t_0+1}\int_{S^3}|\tilde{\nabla}^2u|^2+|\tilde{\nabla}u|^2d\theta dt\leq C\varepsilon^2 \left( e^{- \frac{1}{2}(t(\delta) -t_0)}+ e^{- \frac{1}{2}(t_0-t(\lambda_i R))} \right),
\end{equation*}
Here $\tilde{\nabla}$ is the gradient of $[t(\lambda_i  R), t(\delta)]\times S^3$ with the product metric. Recall that $\abs{\tilde{t}(r)-t(r)}$ is bounded by some universal constant and $\partial_t$ and $\partial_{\tilde{t}}$ are comparable. Hence, we can translate the above decay estimate into a decay with respect to $\tilde{t}=\log r$.
\begin{equation*}
	\int_{\tilde{t}_0-1}^{\tilde{t}_0+1}\int_{S^3}|\tilde{\nabla}^2u|^2+|\tilde{\nabla}u|^2d\theta d\tilde{t}\leq C\varepsilon^2 \left( e^{- \frac{1}{2}(\log (\delta) -\tilde{t}_0)}+ e^{- \frac{1}{2}(\tilde{t}_0-\log(\lambda_i R))} \right),
\end{equation*}

Direct computation shows that
\begin{eqnarray*}
	\int_{B_{e^{\tilde{t}_0+1}}\setminus B_{e^{\tilde{t}_0-1}}}|\nabla^2u|^2+\frac{1}{|x|^2}|\nabla u|^2dx
	&\leq&C\int_{\tilde{t}_0-1}^{\tilde{t}_0+1}\int_{S^3}|\tilde{\nabla}^2u|^2+|\tilde{\nabla}u|^2d\theta d\tilde{t} \\
	&\leq&C\varepsilon^2 \left( e^{-\frac{1}{2} (\log \delta -\tilde{t}_0)}+ e^{- \frac{1}{2}(\tilde{t}_0-\log \lambda_i R)} \right).
\end{eqnarray*}

Then by Sobolev embedding and the $\varepsilon_0-$regularity (Theorem \ref{thm:regularity}), we have
\begin{eqnarray*}
	&&osc_{((\tilde{t}_0-1/2,\tilde{t}_0+1/2)\times S^3)}u \\
	&\leq& C(\int_{B_{e^{\tilde{t}_0+1}}\setminus B_{e^{\tilde{t}_0-1}}}|\nabla^2u|^2+
	\frac{1}{|x|^2}|\nabla u|^2dx)^{1/2}+e^{4\tilde{t}_0(1-1/p)}\|\tau(u)\|_{L^p(B_{e^{\tilde{t}_0+1}}\setminus B_{e^{\tilde{t}_0-1}})} \\
	&\leq& C\varepsilon \left( e^{- \frac{1}{4}(\log \delta -\tilde{t}_0)}+ e^{- \frac{1}{4}(\tilde{t}_0-\log \lambda_i R)} \right).
\end{eqnarray*}

It is easy to derive the no neck estimate from here. Hence, we complete the proof of Theorem.

\section{bounding width by energy}\label{sec:width}
Let $M$ be the manifold in the Theorem \ref{thm:main} and $g$ be any Riemannian metric on $M$. Since $M=M'\# T^m$, there is an embedded sphere $S$ of dimension $m-1$ in $M$ which separates $M$ into $M_1$ and $M_2$ and $M/ M_2$ is homeomorphic to $M'$ and $M/M_1$ is homeomorphic to $T^m$. Here $M/M_i$ is the quotient topology space by identifying all points in $M_i$ as one point.

Let $\tilde{M}$ be a cover of $M$ and $\tilde{g}$ be the lift of $g$. For a map $u:S^4 \to M$, we define the width of $u$ as
\begin{equation*}
	W(u)= \max_{x,y\in S^4} d_{(\tilde{M},\tilde{g})} ( \tilde{u}(x), \tilde{u}(y))
\end{equation*}
for a lift $\tilde{u}$ of $u$.
Since the lift is unique up to the action of the deck transformation of $\tilde{M}$, the definition is independent of the choice of $\tilde{u}$.

\begin{rem}
	It is perhaps more natural to use the universal cover. Theoretically, any cover will make the proof work. Since the main purpose is to construct examples, we use the definition which is convenient for our purpose. Of cause, the width depends on the choice of the cover.
\end{rem}

Similarly, we can define the width of $u$ from $\Real^4$ to $M$ by
\begin{equation*}
	W(u)=\sup_{x,y\in \Real^4} d_{(\tilde{M},\tilde{g})} (\tilde{u}(x), \tilde{u}(y))
\end{equation*}
for a lift $\tilde{u}$.

\begin{rem}
	Since $\Real^4$ is non-compact, it is possible that $W(u)$ is not finite. For application in this paper, we shall only be interested in the bubble map $u:\Real^4 \to M$. There are several ways to see that for a bubble map with finite energy this width is finite. First, one can compose $u$ with the stereographic projection and prove a removable singularity theorem for a PDE system similar but not identical to the biharmonic map equation as Wang did for quasi-biharmonic maps in Lemma 3.4 \cite{Wsphere}. Second, the proof of removable singularity theorem in \cite{ours} can be applied in this case. Finally, since all such bubble maps come from the limit of some biharmonic map sequence, as remarked near the end of Section 2 of \cite{ours}, this is a consequence of the main theorem in \cite{ours}.
\end{rem}

The main result of this section is
\begin{lem}
	\label{lem:finiteR}
	For any $C_1>0$, there is another constant $C_2$ depending  on $C_1$ and the geometry of $M$ such that  any biharmonic map $u$ from $\Real^4$ (or $S^4$ ) to $M$ with $\mathcal E(u)<C_1$ satisfies that $W(u)<C_2$.	
\end{lem}

The proof uses the compactness properties of biharmonic maps (taking the bubbling into account). The non-compactness of $\Real^4$ causes some technical problem. We need the following lemma to control the energy decay at the infinity.
\begin{lem}
	\label{lem:infinity}
	There is a constant $\varepsilon_2>0$ depending on $M$. If $u:\Real^4\to M$ is a biharmonic map satisfying
	\begin{equation*}
		\int_{\Real^4 \setminus B_1} \abs{\nabla^2 u}^2 + \abs{\nabla u}^4 dx <\varepsilon_2,
	\end{equation*}
	then $u$ is uniformly continuous at the infinity in the sense that for any $\varepsilon>0$, there is $R>0$ independent of $u$ such that
	\begin{equation*}
		\mbox{osc}_{\Real^4 \setminus B_R} u <\varepsilon.
	\end{equation*}
\end{lem}

\begin{proof}
	The proof is just another version of Section 6 of \cite{ours}. The only difference is that for a removable singularity theorem, we study $B_1\setminus \set{0}$, which is
	\begin{equation*}
		B_1\setminus \set{0}=\bigcup_{i=1}^\infty A_i
	\end{equation*}
	where
	\begin{equation*}
		A_i=B_{e^{-(i-1)L}}\setminus B_{e^{-iL}},
	\end{equation*}
	while in this lemma, we study the asymptotic behavior of $u$ on
	\begin{equation*}
		\Real^4\setminus B_1= \bigcup_{i=-\infty}^{0} A_i.
	\end{equation*}
	In the proof of the removable singularity theorem, we prove exponential decay as $i\to \infty$ ($\abs{x}\to 0$), while here we prove exponential decay as $i\to -\infty$ ($\abs{x}\to \infty$). We need $\varepsilon_2$ to be small, so that we can use Theorem \ref{thm:tcapprox} on $A_{i-1}\cup A_{i}\cup A_{i+1}$ for $i=-1,-2,\cdots$.

	This lemma follows from the exponential decay of $\abs{\nabla_{S^3} u}$ and $\abs{\partial_t u}$.
\end{proof}

\begin{proof} [Proof of Lemma \ref{lem:finiteR}]
	We only prove the case for $\Real^4$ and the case for $S^4$ is simpler.
	If the lemma is not true, we can find a sequence of biharmonic maps $u_k:\Real^4\to M$ with $\mathcal E(u_k)\leq C_1$, but
	\begin{equation*}
		\lim_{k\to \infty} W(u_k)=+\infty.
	\end{equation*}
	Since $\mathcal E(u_k)$ and $W(u_k)$ are invariant under the scaling,
	we may assume without loss of generality that
	\begin{equation}\label{eqn:a1}
		\int_{\Real^4\setminus B_1} \abs{\nabla^2 u_k}^2 +\abs{\nabla u_k}^4 dx <\varepsilon_2.
	\end{equation}
	(\ref{eqn:a1}) implies that the bubble points are restricted to $\bar B_1$.

	Let $u_\infty$ be the weak limit. Since there is no bubble outside $\bar B_1$, $u_k$ converges to $u_\infty$ on $B_R\setminus B_2$ uniformly for fixed $R$. Together with Lemma \ref{lem:infinity} and (\ref{eqn:a1}), the convergence is uniform on $\Real^4\setminus B_2$.
	

The bubbles are described as follows. Assume that there are $l$ bubbles (including ghost bubbles, which is just trivial map), $\omega_i(i=1,\cdots,l)$ and there are $m(m\leq l)$ blow-up points $p_i(i=1,\cdots,m)$ with $p_i\subset B_2$. Each $\omega_i$ is the limit of
\begin{equation*}
	w_{i,k}(x):=u_k(\lambda_{i,k} x + x_{i,k}).
\end{equation*}
Since there could be bubbles on top of $\omega_i$, the convergence is strong on the domain
\begin{equation*}
	\Omega_{i,k} =B_R \setminus \left( \bigcup_s B_\delta(y_{k,s}) \right),
\end{equation*}
where we use $s$ to parameterize the bubbles on top of $\omega_i$.
Moreover, for each bubble $\omega_i$, there is a neck region of the form $B_{r_2}(*)\setminus B_{r_1}(*)$, which we denote by $N_{i,k}$. There is no need to be precise about $r_1$, $r_2$ and the centers of the balls, it suffices to notice that the no neck theorem implies that
\begin{equation}\label{eqn:noneck}
	\lim_{k\to \infty} \mbox{osc}_{N_{i,k}} u_k = o(\delta,R),
\end{equation}
where $o(\delta,R)$ goes to zero when $\delta\to 0$ and $R\to \infty$.

By definition, if $\tilde{u}_k$ is a lift of $u_k$, we have
\begin{eqnarray}\label{eqn:W}
	W(u_k)&=& \sup_{y,z\in \Real^4} d_{ (\tilde{M},\tilde{g})} (\tilde{u}_k(y), \tilde{u}_k(z)) \\ \nonumber
	&\leq&\sum_{i=1}^l \sup_{y,z\in \Omega_{i,k}} d_{(\tilde{M},\tilde{g})}(\tilde{u}_k(\lambda_{i,k}y+x_{i,k}),\tilde{u}_k(\lambda_{i,k}z +x_{i,k})) \\\nonumber
	&& + \sup_{y,z\in N_{i,k}} d_{(\tilde{M},\tilde{g})}(\tilde{u}_k(y),\tilde{u}_k(z)) \\\nonumber
	&& + \sup_{y,z\in \Real^4 \setminus \bigcup_{i=1}^m B_\delta(p_i)} d_{(\tilde{M},\tilde{g})}(\tilde{u}_k (y), \tilde{u}_k(z)).
\end{eqnarray}

Now we give an upper bound for the left hand side of the above equation.
For the first line, since $w_{i,k}$ converges strongly to $\omega_i$ on $\Omega_{i,k}$, we have
\begin{equation*}
	\max_{y\in \Omega_{i,k}} d_{(M,g)} (w_{i,k}(y), \omega_i(y))\leq o(1).
\end{equation*}
Here $o(1)$ goes to zero as $k\to \infty$.
Noticing that $\tilde{u}_k(\lambda_{i,k}x+x_{i,k})$ is a lift of $w_{i,k}(x)$ (defined on $\Omega_{i,k}$), we can find a lift of $\omega_i$, denoted by $\tilde{w}_i$ such that
\begin{equation*}
	\max_{y\in \Omega_{i,k}} d_{(\tilde{M},\tilde{g})} (\tilde{u}_{k}(\lambda_{i,k}y+x_{i,k}), \tilde{\omega}_i(y))\leq o(1).
\end{equation*}
Therefore, we have
\begin{equation}\label{eqn:w1}
	\limsup_{k\to \infty}\sup_{y,z\in \Omega_{i,k}} d_{(\tilde{M},\tilde{g})}(\tilde{u}_k(\lambda_{i,k}y+x_{i,k}),\tilde{u}_k(\lambda_{i,k}z +x_{i,k}))\leq W(\omega_i).
\end{equation}

For the second line, we need some general fact from Riemannian geometry as follows. There is some small $\sigma>0$ depending on both $(M,g)$ and $(\tilde{M},\tilde{g})$ such that for any geodesic ball $B\subset M$ of radius $\sigma$ and its lift $\tilde{B}\subset \tilde{M}$, we have that $(B, d_{(M,g)})$ is isometric to $(\tilde{B}, d_{(\tilde{M},\tilde{g})})$ as metric spaces.

Thanks to (\ref{eqn:noneck}), for small $\delta$ and large $R$ so that the image $u_k(N_{i,k})$ lies in a geodesic ball of radius $\sigma$, we have
\begin{equation}\label{eqn:w2}
	\limsup_{k\to \infty}  \sup_{y,z\in N_{i,k}} d_{(\tilde{M},\tilde{g})}(\tilde{u}_k(y),\tilde{u}_k(z))\leq Co(\delta,R).
\end{equation}

To bound the last line in (\ref{eqn:W}), it suffices to note that
$u_k$ converges uniformly on $\Real^4 \setminus \bigcup_m B_\delta(p_i)$ to $u_\infty$.
To see this, we note that $u_k$ converges strongly on $B_2\setminus \cup B_{\sigma}(p_i)$ and $u_k$ converges strongly on $\Real^4 \setminus B_2$ as remarked earlier.
Hence,
\begin{equation}\label{eqn:w3}
	\limsup_{k\to \infty}\sup_{y,z\in \Real^4 \setminus \bigcup_{i=1}^m B_\delta(p_i)} d_{(\tilde{M},\tilde{g})}(\tilde{u}_k (y), \tilde{u}_k(z)) \leq W(u_\infty).
\end{equation}

(\ref{eqn:w1}), (\ref{eqn:w2}) and (\ref{eqn:w3}) add up to give an upper bound for $W(u_k)$, which contradicts the assumption that $\lim_{k\to \infty} W(u_k)=\infty$ and hence proves the lemma.
\end{proof}

\section{proof of the main theorem} \label{sec:proof}
Let $u(t)$ be a solution to (\ref{eqn:flow}) with $u(0)=u_0$. Along the flow,
\begin{equation*}
	\frac{d}{dt} E(u)\leq 0.
\end{equation*}
Hence, $E(u)$ is uniformly bounded (before the possible blow-up at least). Since the target manifold is compact, $u$ is bounded and hence $\mathcal E(u)$ is also uniformly bounded.

The key observation to the proof is that for some $C_1>0$ and arbitrarily large $C_3$, we can choose $u_0$ with $\mathcal E(u_0)<C_1$ and any smooth $u'$ homotopic to $u_0$ satisfies $W(u')>C_3$.

Assuming that such $u_0$ is found, we claim that $u(t)$ must blow-up in finite time and hence Theorem \ref{thm:main} is proved. If otherwise, the solution exists for any $t>0$. Since
\begin{equation*}
	\int_0^\infty \int_{S^4} \abs{\partial_t u}^2 dv dt<\infty,
\end{equation*}
we may choose a sequence of $t_k$ going to $\infty$ such that
\begin{equation*}
	\lim_{k\to \infty}	\norm{\partial_t u}_{L^2}(t_k)\to 0.
\end{equation*}
For simplicity, we denote $u(t_k)$ by $u_k$.

Since $\mathcal E(u_k)$ is bounded and the $\varepsilon-$regularity (Theorem \ref{thm:regularity}) holds, the usual blow-up analysis works. Assume that there are $l$ bubbles $\omega_i(i=1,\cdots,l)$, which is the limit of $u_k(\lambda_{i,k} x + x_{i,k})$ and $m(m<l)$ blow-up points $p_i$. Let $\Omega_{i,k}$ and $N_{i,k}$ as before. We still have
\begin{eqnarray*}
	W(u_k)&\leq& \sum_{i=1}^l \sup_{y,z\in \Omega_i} d_{(\tilde{M},\tilde{g})}(\tilde{u}_k(\lambda_{i,k}y+x_{i,k}),\tilde{u}_k(\lambda_{i,k}z+x_{i,k})) \\
	&& + \sup_{y,z\in N_{i,k}} d_{(\tilde{M},\tilde{g})}(\tilde{u}_k(y),\tilde{u}_k(z)) \\
	&& + \sup_{y,z\in \Real^4 \setminus \bigcup_{i=1}^m B_\delta(p_i)} d_{(\tilde{M},\tilde{g})}(\tilde{u}_k (y), \tilde{u}_k(z)).
\end{eqnarray*}
By Theorem \ref{thm:noneck}, we can bound the right hand side by
\begin{equation*}
	\sum_{i=1}^l W(\omega_i) + W(u_\infty)+1.
\end{equation*}
By Lemma \ref{lem:finiteR}, each $W(\omega_i)$ and $W(u_\infty)$ is bounded by a constant $C_2$ depending on $C_1$. Moreover, the number of bubbles is also bounded by a constant depending on $C_1$. Hence, there is a constant $C_4$ such that
\begin{equation*}
	\limsup_{k\to \infty} W(u_k) < C_4.
\end{equation*}
This would be a contradiction and hence proves Theorem \ref{thm:main} if $C_4<C_3$.

Now let's show how to construct $u_0$.

Recall that $M=M'\# T^m$. There is a natural cover of $M$, which is obtained by modifying $\Real^m$. $\Real^m$ is the universal cover of $T^m$, with the deck transformation group $G=\mathbb Z^m$. Let $p_0$ be any point of $\Real^m$ and let the orbit of the action of $G$ containing $p_0$ be $\set{p_i}_{i=0}^\infty$. Suppose $U_i$ be a small neighborhood of $p_i$ diffeomorphic to the ball of dimension $m$ and $V\subset M'$ be an open set diffeomorphic to a ball. For each $i=0,1,\cdots$, we remove $U_i$ from $\Real^m$ and identify the boundary of $U_i$ with the boundary of a copy of $M'\setminus V$, which we denote by $W_i$. The new complete non-compact manifold is denoted by $\tilde{M}$. $G$ acts on $\tilde{M}$ naturally and the quotient is $M$. If $M$ is equipped with a Riemannian metric $g$ and $\tilde{g}$ is the pull back metric, then the projection $\pi:\tilde{M}\to M$ is isometric map.

Since $\pi_4(M')$ is not trivial, there is a smooth map $h:S^4\to M'$, which is not homotopic to constant map. Since $m>4$ and $h$ is not surjective, assume by deforming it smoothly that

(1) $h(S^4)\subset M'\setminus \bar{V}$;

(2) $h$ maps the entire southern hemisphere to a single point $q\in M'\setminus \bar{V}$.

Let $h_i$ be the copy of $h$ from $S^4$ to $W_i$ and $q_i$ be the copy of $q$ in $W_i$.

For any $C_3$, pick $i$ such that
\begin{equation*}
	d_{(\tilde{M},\tilde{g})}(W_0,W_i) >C_3.
\end{equation*}
Let $\Psi (\Phi)$ be the stereoprojection from $\Real^4$ to $S^4$, which maps the infinity to the south (north) pole and maps $\partial B_1$ to the equator. Consider the map $w:\Real^4\to W_i$ defined by
\begin{equation*}
	w(x)=h_i \circ \Psi (x).
\end{equation*}
$w$ is a constant map outside $B_1$. Set
\begin{equation*}
	C_1= E(w) + E(h) +1.
\end{equation*}
We claim that for $\sigma$ very small, we can find smooth $u_0$ satisfying

(1)
\begin{equation*}
	u_0=
	\left\{
		\begin{array}[]{ll}
			\pi\circ h_0(x) & \quad x\in S^4\setminus B_\sigma(S);\\
			\pi\circ w(\frac{\Phi^{-1}(x)}{\sigma^2/2}) & \quad x\in B_{\sigma^2}(S).
		\end{array}
		\right.
\end{equation*}

(2) $E(u_0)< C_1$.

By the above definition, we observe that $u_0|_{\partial B_\sigma (S)}=q_0$ and $u_0|_{\partial B_{\sigma^2}(S)}=q_i$. The first observation follows trivially from the definition of $h_0$. For the latter, we notice that $\abs{\Phi^{-1}(x)}$ is almost $\sigma^2$ for every $x\in \partial B_{\sigma^2}(S)$, because $\Phi$ is almost an isometry near $S$ and $\sigma$ is going to be small.

Since the energy is scaling invariant and $\Phi$ is almost isometric in small neighborhood of the south pole, we have
\begin{equation*}
	\left(\int_{S^4\setminus B_\sigma(S)} + \int_{B_{\sigma^2}(S)}\right) \abs{\triangle u_0}^2 dv < E(h)+ E(w)+\frac{1}{2}.
\end{equation*}

It suffices to show that we can define $u_0$ on $B_\sigma(S)\setminus B_{\sigma^2}(S)$ so that $u_0$ is smooth and the contribution to the energy on this part is smaller than $\frac{1}{2}$. By choosing $\sigma$ small, the metric of $S^4$ on $B_\sigma$ is close to the flat metric. Hence, it suffices to check this with flat metric.

Let $\gamma:[0,1]\to \tilde{M}$ be the shortest geodesic in $\tilde{M}$ connecting $q_0$ to $q_i$. Let $\varphi:[0,1] \to [0,1]$  be a smooth function satisfying

(1) $\varphi'\geq 0$;

(2) $\varphi(x)=0$ for all $0\leq x\leq \frac{1}{8}$ and $\varphi(x)=1$ for all $\frac{7}{8}\leq x\leq 1$;

(3) $\abs{\varphi'}+ \abs{\varphi''}\leq C$ for some universal constant $C$.

Set
\begin{equation*}
	u_0(x)= \pi\circ \gamma\circ \varphi \left( \frac{\log \sigma- \log \abs{x}}{ - \log \sigma} \right).
\end{equation*}
For simplicity, we write $L$ for $d_{(\tilde{M},\tilde{g})}(q_0,q_i)$.
Note that
\begin{equation*}
	\abs{(\pi\circ \gamma)'}= {L}.
\end{equation*}
Since $\gamma$ and $\pi\circ \gamma$ are geodesics, we have
\begin{equation*}
(\pi\circ \gamma)''+ B(\pi\circ \gamma)( (\pi\circ \gamma) ',(\pi\circ \gamma)')=0	
\end{equation*}
where $B$ is the second fundamental form of $N$. Therefore,
\begin{equation*}
	\abs{(\pi\circ \gamma)''}= {C}{L^2}.
\end{equation*}
We estimate the derivative of $u_0$ as follows.
\begin{equation*}
	\abs{\partial_r u_0}\leq \frac{CL}{ r (-\log \sigma)}
\end{equation*}
and
\begin{equation*}
	\abs{\partial_r^2 u_0}\leq \frac{CL}{ r^2 (-\log \sigma)}.
\end{equation*}
Hence,
\begin{eqnarray*}
	&& \int_{B_\sigma\setminus B_{\sigma^2}} \abs{\triangle u_0}^2 dx \\
	&\leq& C \int_{\sigma^2}^\sigma \abs{\partial_r^2 u_0 +\frac{3}{r}\partial_r u_0}^2 r^3 dr \\
	&\leq& \frac{CL^2 }{(\log \sigma)^2} \int_{\sigma^2}^\sigma \frac{1}{r} dr \\
	&\leq& \frac{CL^2 }{(-\log \sigma)}.
\end{eqnarray*}
For any $L$, we can choose $\sigma$ so that the above is as small as we want. Hence, we check that $u_0$ satisfies $E(u_0)< C_1$. It remains to check that for any map $u'$ homotopic to $u_0$, $W(u')> C_3$. Let $\tilde{u}'$ be the lift of $u'$, which is homotopic to the following lift of $u_0$,
\begin{equation*}
	\tilde{u}_0=\left\{
		\begin{array}[]{ll}
			h_0(x) & \quad x\in S^4\setminus B_\sigma(S);\\
			\gamma\circ \varphi (\frac{\log \sigma -\log \abs{x}}{-\log \sigma}) & \quad x\in B_\sigma(S)\setminus B_{\sigma^2}(S); \\
			w(\frac{\Phi^{-1}(x)}{\sigma^2/2}) & \quad x\in B_{\sigma^2}(S).
		\end{array}
		\right.
\end{equation*}
We claim that $\tilde{u'}\cap W_0\ne \emptyset$ and $\tilde{u'}\cap W_i\ne \emptyset$. To see this, consider a continuous map $\tilde{\pi}$ from $\tilde{M}$ to $M'$ (precisely, a manifold homeomorphic to $M'$), which maps any point in $M\setminus W_0$ to one point. If $\tilde{u}'\cap W_0$ is empty, then $\tilde{\pi}\circ \tilde{u}'$ is a constant map. However, $\tilde{\pi}\circ \tilde{u}_0$ is homotopic to $h_0$ and hence is nontrivial. The proof for $\tilde{u'}\cap W_i\ne \emptyset$ is the same.

In summary, we have constructed a map $u_0$ such that $E(u_0)<C_1$ and $W(u')>C_3$ for any $u'$ homotopic to $u_0$. This finishes the proof of Theorem \ref{thm:main}.

\end{document}